\documentclass{article}
\usepackage{graphicx}
\usepackage{amsthm, amsmath, amssymb, tikz, bbm}
\usepackage{mathtools}  
\usepackage{xifthen}

\usepackage[shortlabels]{enumitem}

\usetikzlibrary{calc,shapes, backgrounds}

\newtheorem{theorem}{Theorem}

\newtheorem{proposition}{Proposition}

\newcommand{\ds}{\displaystyle}
\newcommand{\dss}{\displaystyle\sum}

\newcommand{\E}{{\rm E}}
\newcommand{\lp}{\left (}
\newcommand{\rp}{\right )}

\newcommand{\cB}{\mathcal{B}}

\newcommand{\cH}{\mathcal{H}}

\newcommand{\cG}{\mathcal{G}}
\newcommand{\cP}{\mathcal{P}}

\newcommand{\cT}{\mathcal{T}}
\newcommand{\dP}{\cP_{n,k-1}^{(k)}}
\newcommand{\dlP}{\cP_{n,l}^{(k)}}

\DeclarePairedDelimiter{\ceil}{\lceil}{\rceil}

\pgfdeclarelayer{edgelayer}
\pgfdeclarelayer{nodelayer}
\pgfsetlayers{edgelayer,nodelayer,main}

\begin{document}

\title{On the size-Ramsey number of tight paths}
\author{
Linyuan Lu
\thanks{University of South Carolina, Columbia, SC 29208,
({\tt lu@math.sc.edu}). This author was supported in part by NSF
grant DMS-1600811.} \and
Zhiyu Wang \thanks{University of South Carolina, Columbia, SC 29208,
({\tt zhiyuw@math.sc.edu}).} 
}

\maketitle

\begin{abstract}
  For any $r\geq 2$ and $k\geq 3$, the $r$-color size-Ramsey number $\hat R(\mathcal{G},r)$ of a $k$-uniform hypergraph $\mathcal{G}$
  is the smallest
  integer $m$ such that there exists a $k$-uniform hypergraph $\mathcal{H}$ on $m$ edges such that any coloring
  of the edges of $\mathcal{H}$ with $r$ colors yields a monochromatic copy of $\mathcal{G}$.
  Let $\mathcal{P}_{n,k-1}^{(k)}$ denote the $k$-uniform tight path on $n$ vertices.
Dudek, Fleur, Mubayi and R\H{o}dl showed that the size-Ramsey number of tight paths 
$\hat R(\mathcal{P}_{n,k-1}^{(k)}, 2) = O(n^{k-1-\alpha} (\log n)^{1+\alpha})$ where $\alpha = \frac{k-2}{\binom{k-1}{2}+1}$. 
 In this paper, we improve their bound by showing that
 $\hat R(\mathcal{P}_{n,k-1}^{(k)}, r) = O(r^k (n\log n)^{k/2})$
 for all $k\geq 3$ and $r\geq 2$.

\end{abstract}

\begin{section}{Introduction}

Given two simple graphs $G$ and $H$ and a positive integer $r$, say that $H \to (G)_r$ if every $r$-edge-coloring of $H$ results in a monochromatic copy of $G$ in $H$. In this notation, the Ramsey  number $R(G)$ of G is the minimum $n$ such that $K_n \to (G)_2$. The {\em size-Ramsey number} $\hat{R}(G,r)$ of $G$ is defined as the minimum number of edges in a graph $H$ such that $H \to (G)_r$, i.e.
\[\hat{R}(G,r) = min\{|E(H)|: H \to (G)_r\}.\]
When $r = 2$, we ignore $r$ and simply use $\hat{R}(G)$.

Size-Ramsey number was first studied by Erd\H{o}s, Faudree, Rousseau and Schelp \cite{Erdos} in 1978. By the definition of $R(G)$, we have $$\hat{R}(G) \leq \binom{R(G)}{2}.$$
Chv\'atal (see, e.g.\cite{Erdos}) showed that this bound is tight for complete graphs, i.e. $\hat{R}(K_n) = \binom{R(K_n)}{2}.$ Answering a question of Erd\H{o}s \cite{Erdos2}, Beck \cite{Beck} showed by a probabilistic construction that  
\[\hat{R}(P_n) = O(n).\]
Alon and Chung \cite{Alon-Chung} gave an explicit construction of a graph $G$ with $O(n)$ edges such that $G \to P_n$. Recently, Dudek and Pra\l at \cite{Dudek-Pralat} provided a simple alternative proof for this result (See also \cite{Letzter}).
The best upper bound $\hat{R}(P_n)\leq 74n$ is due to  Dudek and Pra\l at \cite{Dudek-Pralat2}
by considering a random $27$-regular graph of a proper order. 

 Dudek, Fleur, Mubayi, and R\H{o}dl \cite{DFMR} first initiated the study of size-Ramsey number in hypergraphs. A $k$-uniform hypergraph $\cG$ on a vertex set $V(\cG)$ is a family of $k$-element subsets (called edges) of $V(\cG)$. We use $E(\cG)$ to denote the edge set. Given $k$-uniform hypergraphs $\cG$ and $\cH$, we say that $\cH \to (\cG)_r$ if every $r$-edge-coloring of $\cH$ results in a monochromatic copy of $\cG$ in $\cH$. Define the {\em size-Ramsey number} $\hat{R}(\cG,r)$ of a $k$-uniform hypergraph $\cG$ as 
 \[\hat{R} (\cG, r) = min\{|E(\cH)|: \cH \to (\cG)_r\}.\]
When $r = 2$, we simply use $\hat{R} (\cG)$ for the ease of reference.

 Given integers  $1\leq l < k$ and $n \equiv l$ (mod $k-l$), an $l$-path $\dlP$ is a $k$-uniform hypergraph with vertex set $[n]$ and edge set $\{e_1, \cdots, e_m\}$, where $e_i = \{(i-1)(k-l)+1, (i-1)(k-l)+2, \cdots, (i-1)(k-l)+k\}$ and $m = \frac{n-l}{k-l}$, i.e. the edges are intervals of length $k$ in $[n]$ and consecutive edges intersect in exactly $l$ vertices. A $\cP_{n,1}^{(k)}$ is commonly referred as a {\em loose} path and a $\cP_{n,k-1}^{(k)}$ is called a {\em tight} path.

 Dudek, Fleur, Mubayi and R\H{o}dl \cite{DFMR} showed that when $l\leq \frac{k}{2}$, the size-Ramsey number of a path $\cP_{n,l}^{(k)}$ can be easily reduced to the graph case. In particular, they showed that if $1\leq l\leq \frac{k}{2}$, then 
\[\hat{R}\lp \cP_{n,l}^{(k)} \rp \leq \hat{R}(P_n) = O(n).  \]

For tight paths, they showed in the same paper that for fixed $k \geq 3$,
\[\hat{R}\lp \dP \rp = O(n^{k-1-\alpha} (\log n)^{1+\alpha}),\]
where $\alpha = (k-2)/(\binom{k-2}{2}+1)$. Observe that $\hat{R}\lp \cP_{n,l}^{(k)} \rp \leq \hat{R}\lp \dP \rp$. Thus any upper bound on the size-Ramsey number of tight paths is also an upper bound for
other $l$-path $\cP_{n,l}^{(k)}$.

Motivated by their approach, we use a different probabilistic construction
and improve the upper bound to $O((n\log n)^{k/2})$. In particular, we show the following result on the multi-color size-Ramsey number of tight paths in hypergraphs:

\begin{theorem}\label{tightpath}
 For any fixed $k \geq 3$, any $r \geq 2$, and sufficiently large $n$, we have
 \[\hat{R}\lp \dP, r\rp = O\lp r^{k}  (n \log n)^{\frac{k}{2}}\rp.\]
\end{theorem}

\end{section}

\begin{section}{Proof of Theorem \ref{tightpath}}
 
The approach of our proof is inspired by Dudek, Fleur, Mubayi and R\H{o}dl's approach in their proof of Theorem 2.8 in \cite{DFMR}. In their proof, they constructed their hypergraph by setting edges to be the $k$-cliques of an Erd\H{o}s-R\'{e}nyi random graph. Then they use a greedy algorithm to show that the number of edges of each color is smaller than $\frac{1}{r}$ fraction of the total number of edges, which gives a contradiction. Motivated by their approach, we use the same greedy algorithm but a different probabilistic construction of the hypergraph. Instead of using $k$-cliques of an Erd\H{o}s-R\'{e}nyi random graph as edges, we use $k$-cycles of a random $C_k$-colorable graph (which will be defined later) as edges.

Throughout the paper, we will use the following version of Chernoff inequalities for the binomial random variables $X \sim Bin(n,p)$ (for details, see, e.g. \cite{Chernoff}):

\begin{equation}\label{eq:chernoff1}
Pr \lp X \leq E(X) -\lambda \rp \leq exp \lp -\frac{\lambda^2}{2 E(X)} \rp.
\end{equation}

\begin{equation}\label{eq:chernoff2}
Pr \lp X \geq E(X) + \lambda \rp \leq exp \lp -\frac{\lambda^2}{2 (E(X)+ \lambda/3)} \rp.
\end{equation}

We follow  a similar notation as \cite{DFMR}. A graph $G$ is {\em $C_k$-colorable}
if there is a graph homomorphism $\pi$ mapping $G$ to the cycle $C_k$. That is, $V(G)$ can be partitioned
into $k$-parts $V_1\cup V_2\cup \cdots \cup V_k$ so that $E(G)\subseteq\bigcup\limits_{i=1}^k E(V_i, V_{i+1})$
with $V_{k+1}=V_1$ and $E(V_i, V_{i+1})$ denoting the set of edges between a vertex in $V_i$ and a vertex in $V_{i+1}$. 
For such a graph $G$, we say a $k$-cycle $C$ in $G$ is {\em proper} if it intersects each $V_i$
by exactly one vertex.  For $1\leq l\leq k-1$, we say a path $P_l$ of $l$ vertices in $G$ is {\em proper}
if it intersects each $V_i$ by at most one vertex.
Let $\cT_{k-1}(G)$ denote the set of all proper $(k-1)$-paths in $G$.
Let $\cB \subseteq \cT_{k-1}$ be a family of pairwise vertex-disjoint proper
$(k-1)$-paths. Let $t_{\cB}$ be the total number of proper $k$-cycles in $G$ that extend some $B\in \cB$. For $A\subseteq V$,  define $y_{A, \cB}$ as the number of proper $k$-cycles in $G$ that
extend a proper $(k-1)$-path $B\in \cB$ with a vertex $v\in A\cup \bigcup_{B\in \cB} V(B)$.
Given $C\subseteq V(G)$, we use $z_C$ to denote the number of proper $k$-cycles in $G$ that intersect $C$. We use $t_k$ to denote the total number of proper $k$-cycles in $G$.

We say an event in a probability space holds {\em a.a.s.}\ (aka, {\em asymptotically almost surely}) if the probability that it holds tends to $1$ as $n$ goes to infinity. Finally, we use $\log n$ to denote natural logarithms.

\begin{proposition}\label{prop2}
  For every $r \geq 2$, $k \geq 3$, and sufficiently large $n$, there exists a $C_k$-colorable graph
  $G = (V,E)$ satisfying the following:
\begin{enumerate}[(i)]
\item For every $\cB$ consisting of $n$ pairwise vertex-disjoint
  proper $(k-1)$-paths, and every $A \subseteq V\setminus \bigcup_{B\in \cB} V(B)$ with 
$|A|\leq n$,   we have 
		\[y_{A,\cB} < \frac{1}{2kr} t_{\cB}.\]
	\item For every $C\subseteq  V$ with  $|C| \leq (k-1)n$, we have
		\[z_C< \frac{t_k}{2r}.\]
	\item The total number of proper $k$-cycles satisfies 
		\[t_k =  O(r^{k} (n\log n)^{k/2}).\]
\end{enumerate}
\end{proposition}

\begin{proof}
  Set $c = 16k^2r$ and $p=\frac{\sqrt{\log n}}{\sqrt{n}}$.
  Consider the following random $C_k$-colorable graph $G$.
  Let $V(G) = V_1 \cup V_2 \cup \cdots \cup V_{k}$ be the disjoint union of $k$ sets. Each $V_i$ (for $1\leq i\leq k$) has the same size $cn$. For any pair of vertices $\{u,v\}$ in two consecutive parts,  i.e., there is an $i\in [k]$, such that $u\in V_i$ and $v\in V_{i+1}$ (with the convention $V_{k+1}=V_1$), add
  $uv$ as an edge of $G$ with probability $p$ independently. There is no edge
  inside each $V_i$ or between two non-consecutive parts.
  

  We will show that this random $C_k$-colorable graph $G$ satisfies a.a.s. $(i)- (iii)$.

  First we show that $G$ a.a.s. satisfies $(i)$. For a fixed family $\cB$
  of $n$ pairwise vertex-disjoint proper $(k-1)$-paths, we would like to give a lower bound of
  $t_{\cB}$. For each proper $(k-1)$-path $B\in \cB$, there are $cn$ vertices that can extend
  $B$ into a proper $k$-cycle, each with probability $p^2$ independently.  Thus, we have
  $t_{\cB} \sim Bin(cn^2, p^2)$ with 
\[E[t_{\cB}] =cn^2p^2 =c n \log n=16k^2r n\log n. \]
Applying Chernoff inequality, we have 
\begin{align*}
Pr \lp t_{\cB} \leq \frac{E[t_{\cB}])}{2}\rp  & \leq exp \lp  -\frac{1}{8} E[t_{\cB}]    \rp\\
  &=  exp \lp   -2k^2rn \log n\rp.
\end{align*}

Now for fixed $A\subseteq V\setminus \bigcup_{B\in \cB} V(B)$,
we estimate the upper bound of $y_{A,\cB}$.   Without loss of generality, we can assume that $|A| = n$.
We have $y_{A,\cB} \leq Y \sim Bin(2n^2, p^2)$, thus
\[E[Y] = 2n^2 p^2 = 2n \log n.\]

Thus if we apply the Chernoff bound \eqref{eq:chernoff2} with $\lambda = (2k-1) E[Y]$, then
 \begin{align*}
	Pr \lp  Y \geq \frac{1}{4kr}  E[t_{\cB}]\rp &= Pr \lp Y \geq  2k E[Y]\rp \\
                                                             &= Pr \lp Y \geq E[Y] + \lambda \rp \\
                                              &\leq \exp \lp -\frac{\lambda^2}{2(\E[Y]+\lambda/3)} \rp\\
				 & \leq \exp \lp - \frac{3(2k-1)^2}{2k+2} n \log n\rp.
\end{align*}

The number of possible choices of $\cB$ is upper bounded by $\lp  \binom{cn}{n} \cdot n!  \rp^k$. The number of possible choices of $A$ and $\cB$ is upper bounded by $\lp\binom{cn}{n,\ceil{n/k}} \cdot n!\rp^k \leq \lp\binom{cn}{n,n} \cdot n!\rp^k $. Stirling approximation of binomial coefficient gives us that 
\begin{align*}
	\log \lp\binom{cn}{n} \cdot n!\rp^k &= (1+o(1)) \lp k n \log n\rp,\\
	\log \lp\binom{cn}{n,n} \cdot n!\rp^k &= (1+o(1)) \lp k n \log n\rp.
\end{align*}

Therefore by the union bound, we have
\begin{align*}
Pr \lp \bigcup_{\cB} \{ t_{\cB}  \leq \frac{E[t_{\cB}]}{2}\}\rp  & \leq \lp\binom{cn}{n} \cdot n!\rp^k Pr \lp t_{\cB} \leq \frac{E[t_{\cB}]}{2}\rp \\
			& \leq exp\lp   (1+o(1))kn \log n-  2k^2rn \log n\rp\\
			&  = o(1).
\end{align*}

Similarly, we have
\begin{align*}
  Pr \lp \bigcup_{A, \cB} \{ y_{A, \cB}  \geq \frac{1}{4kr} E[t_{\cB}]  \}\rp  & \leq \lp\binom{cn}{n,n} \cdot n!\rp^k Pr \lp  Y \geq \frac{1}{4kr} E[t_{\cB}]\rp \\
	& \leq exp\lp   (1+o(1))kn \log n-  \frac{3(2k-1)^2}{2k+2}n \log n\rp\\  
			&  = o(1).
\end{align*}
In the last step, we observe $\frac{3(2k-1)^2}{2k+2}> k$ for all $k\geq 3$.

Therefore, combining previous inequalities, it follows that for all $A,\cB$ satisfying the condition in $(i)$, we have, a.a.s.,
\[y_{A,\cB} < \frac{1}{4kr} E[t_{\cB}] \leq \frac{1}{2kr} t_{\cB}.\]
This finishes the proof of $(i)$.

Now we will prove that $G$ satisfies $(ii)$ and $(iii)$ a.a.s.

We will use the Kim-Vu inequality \cite{Kim-Vu} stated as below:
\begin{quotation}\it
Let $H$ be a (weighted) hypergraph with $V(H) = [n]$. Edge edge $e$ has some weight $w(e)$. Suppose $\{t_i: i\in [n]\}$ is a set of Bernoulli independent random variables with probability $p$ of being $1$. Consider the polynomial 
\[Y_H = \dss_{e\in E(H)} w(e) \ds\prod_{s\in e} t_s.\]
Furthermore, for  a subset $A$ of $V(H)$, define \[Y_{H_A} = \dss_{e, A\subset e} w(e) \ds\prod_{i\in e\backslash A} t_i.\]
If we define $E_i(H) = \ds\max_{A\subset V(H), |A| = i} E(Y_{H_A})$, $E(H) = \ds\max_{i\geq 0} E_i(H)$ and $E'(H) = \ds\max_{i\geq 1} E_i(H)$, then
\begin{equation}\label{Kim-Vu}
Pr \lp |Y_H- E_0(H)| > a_k (E(H)E'(H))^{1/2} \lambda^k\rp = O\lp exp(-\lambda + (k-1) \log n)\rp
\end{equation} for any positive number $\lambda > 1$ and $a_k = 8^k (k!)^{1/2}$.
\end{quotation}

In our context, for a fixed $v \in V(G)$, let $H$ be the $k$-uniform hypergraph constructed by the proper $k$-cycles of $G$ containing $v$. The edge set of $H$ is the collection of all $k$-tuples $\{v v_1, v_1v_2, \cdots, v_{k-2}v_{k-1}, v_{k-1} v\}$ such that $vv_1v_2 \cdots v_{k-1} v$ is a proper $k$-cycle in $G$ and all edges have weight $1$.

Fix $v\in V(G)$. we let $X_v$ denote the number of proper $k$-cycles in $G$ that contain $v$. Then it's not hard to see that
\[E_0(X_v) = E(X_v) = (cn)^{k-1} p^k = c^{k-1} n^{\frac{k-2}{2}} (\log n)^{\frac{k}{2}}.\]
\[E'(X_v) = (cn)^{k-2} p^{k-1} = c^{k-2} n^{\frac{k-3}{2}}  (\log n)^{\frac{k-1}{2}}.\]

Applying Kim-Vu inequality with $\lambda = 2(k-1) \log n$, we get that for each $v\in V(G)$,

\[ Pr \lp |X_v- E_0(X_v)| > a_k (E(X_v)E'(X_v))^{1/2} \lambda^k\rp = O\lp exp(- (k-1) \log n)\rp.\]

Observe that $a_k (E(X_v)E'(X_v))^{1/2} \lambda^k = o(E_0(X_v))$. Applying union bound for all $v\in V(G)$, we obtain that a.a.s that
 $$X_v  =  (1\pm o(1)) (cn)^{k-1}p^k= (1\pm o(1))c^{k-1}  n^{\frac{k}{2}-1} (\log n)^{\frac{k}{2}}.$$
Recall that $t_k$ denotes the total number of proper $k$-cycles in $G$ and $z_C$ denotes the number of proper $k$-cycles in $G$ that intersect $C$. Suppose $|C| \leq (k-1)n.$ Then 
\[z_C \leq (1+o(1)) (k-1)n c^{k-1}  n^{\frac{k}{2}-1}(\log n)^{\frac{k}{2}}=(1+o(1)) (k-1) c^{k-1} (n \log n)^{\frac{k}{2}}. \]
Note that $t_k = \frac{1}{k}\dss_{v\in V(G)}X_v$. Thus 
\begin{align*}
 t_k  & \geq \frac{1}{k} (1-o(1)) kcn \cdot c^{k-1}  n^{\frac{k}{2}-1} (\log n)^{\frac{k}{2}}  \\
        & \geq (1-o(1)) c^k (n\log n)^{\frac{k}{2}}.
\end{align*}
Since $c= 16k^2r$, we have that for n sufficiently large, 
\[z_C < \frac{t_k}{2r}.\]
Moreover, similar to the above calculation, we have that a.a.s.,
\[t_k \leq (1+o(1)) c^k (n\log n)^{\frac{k}{2}} = O(r^{k} (n\log n)^{\frac{k}{2}}).\]
\end{proof}

Now we will prove the main result. We use the same greedy algorithm approach by Dudek, Fleur, Mubayi and R\H{o}dl in \cite{DFMR}.\\


\begin{proof}[Proof of Theorem \ref{tightpath}:]
We show that there exists a $k$-uniform hypergraph $\cH$ with $|E(\cH)| = O(r^k n^{\frac{k}{2}}(\log n)^{\frac{k}{2}})$ such that any $r$-coloring of the edges of $\cH$ yields a monochromatic copy of $\dP$.

Let $G$ be the graph constructed from Proposition \ref{prop2} for $n$ sufficiently large. Let $\cH$ be a $k$-uniform hypergraph such that $V(\cH) = V(G)$ and $E(\cH)$ be the collection of all proper $k$-cycles in $G$.

Take an arbitrary $r$-coloring of the edges $\cH_0 = \cH$ and assume that there is no monochromatic $\dP$. Without loss of generality, suppose the color class with the most number of edges is blue. We will consider the following greedy algorithm:

\begin{enumerate}[(1)]
	\item  Let $\cB = \emptyset$ be a {\em trash set} of proper $(k-1)$-paths in $G$. Let $A$ be a blue tight path in $\cH$ that we will iteratively modify. Throughout the process, let $U = V(\cH) \backslash \lp V(A) \cup \bigcup_{B\in \cB} V(B)\rp$ be the set of {\em unused} vertices. If at any point $|\cB| = n$, terminate.
	\item If possible, choose a blue edge $v_1 v_2\cdots v_{k-1} v_k$ from $U$ and put these vertices into $A$ and set the pointer to $v_k$. Otherwise, if not possible, terminate.
	\item  Suppose the pointer is at $v_i$ and $v_{i-k+2}, \cdots, v_{i-1}, v_i$ are the last $k-1$ vertices of the constructed blue path $A$. There are two cases:
		\begin{description}
			\item Case 1: If there exists a vertex $u \in U$ such that $v_{i-k+2},\cdots,v_{i-1},v_i,u$ form a blue edge in $\cH$, then we {\em extend} $P$, i.e. add $v_{i+1} = u$ into $A$. Set the pointer to $v_{i+1}$ and restart Step $(3)$. 
			\item Case 2: Otherwise, remove the last $k-1$ vertices from $A$ and set $\cB = \cB \cup \{\{v_{i-k+2}, \cdots, v_{i-1}, v_i\}\}$. Set the pointer to $v_{i-k+1}$. Now if $|A| <k$, then set $A = \emptyset$ and go to Step $(2)$. Otherwise, restart Step $(3)$.
		\end{description}
\end{enumerate}	
	
Note that this procedure will terminate under two circumstances: either $|\cB| = n$ or there is no blue edge in $U$.

Let us first consider the case when $|\cB| = n$, i.e. there are $n$ pairwise vertex-disjoint proper $(k-1)$-paths  in $\cB$. Moreover, $|A| \leq n$ since there is no blue path of $n$ vertices. Applying Proposition $\ref{prop2}$ with sets $A$ and $\cB$, we obtain that  
\[y_{A,\cB} < \frac{1}{2kr} t_{\cB}.\] 
Observe that every edge of $\cH$ that extends  some $B \in \cB$ with a vertex from $V(\cH_0) \backslash \lp V(A) \cup \bigcup\limits_{B\in \cB_m} B\rp$ must be non-blue. Therefore, the number of blue edges of $\cH$ that contain some $B\in \cB$ as subgraph is at most $y_{A,\cB}.$

Consider $A, \cB$ as $A_0, \cB_0$ respectively. Now remove all the blue edges from $\cH_0$ that contain some $B\in \cB_0$ as subgraph and denote the resulting hypergraph as $\cH_1$. Perform the greedy procedure again on $\cH_1$. This will generate a new $A_1$ and $\cB_1$.  Applying Proposition \ref{prop2} again, we have  $y_{A_1, \cB_1} \leq  \frac{1}{2kr} t_{\cB_1}.$  Keep repeating the procedure until it is no longer possible. Observe that $\cB_i \cap \cB_j = \emptyset$ for $i\neq j$.

When the above procedure can not be repeated anymore, we are in the case that $|\cB_m| < n$ for some positive integer $m$ and there are no more blue edges in $V(\cH) \backslash \bigcup\limits_{B\in \cB_m} B$. In this case, $A_m = \emptyset$ and all the blue edges remaining in $\cH_m$ have to intersect the set $C = \bigcup\limits_{B\in \cB_m} B$. By Proposition \ref{prop2}, it follows that
\[z_C < \frac{1}{2r} t_k.\]

Let $e_b(\cH)$ denote the total number of blue edges in $\cH$. We have
\begin{align*}
e_b(\cH)  & \leq  \dss_{i=0}^{m-1}  y_{A_i, \cB_i} +  z_C\\
          & < \dss_{i=0}^{m-1} \frac{1}{2kr} t_{\cB_i}  + \frac{1}{2r} t_k.
\end{align*}
Note that every proper $k$-cycle can extend exactly $k$ proper $(k-1)$-paths.
We have $\dss_{i=0}^{m-1} t_{\cB_i}\leq kt_k$. Thus,
\begin{align*}
e_b(\cH)	     & <  \frac{1}{2kr} \dss_{i=0}^{m-1} t_{\cB_i}  + \frac{1}{2r} t_k\\
	     & \leq  \frac{1}{2r} t_k  + \frac{1}{2r} t_k\\
	     & = \frac{1}{r} |E(\cH)|.
\end{align*}
The conclusion is that the number of blue edges in $\cH$ is strictly smaller than $\frac{1}{r}$ of the total number of edges in $\cH$, which contradicts that blue is the color class with the most number of edges of $\cH$. 
\end{proof}
\end{section}

\end{document}